\documentclass[]{interact}

\usepackage{epstopdf}
\usepackage[utf8]{inputenc}
\usepackage{csquotes}
\usepackage{babel}
\usepackage[caption=false]{subfig}
\usepackage{pdflscape}
\usepackage{lscape}
\usepackage{rotating}
\usepackage{hyperref}
\usepackage{algorithm}
\usepackage{algorithmic}
\usepackage{graphicx}
\usepackage[table]{xcolor}
\usepackage{multirow}
 \usepackage[nameinlink,capitalise]{cleveref}

\usepackage[numbers,sort&compress]{natbib}
\bibpunct[, ]{[}{]}{,}{n}{,}{,}
\makeatletter
\def\NAT@def@citea{\def\@citea{\NAT@separator}}
\makeatother

\theoremstyle{plain}
\newtheorem{theorem}{Theorem}[section]
\newtheorem{lemma}[theorem]{Lemma}

\theoremstyle{definition}

\newtheorem{experiment}[theorem]{Experiment}
\newtheorem{assumption}{Assumption}

\theoremstyle{remark}
\newtheorem{remark}{Remark}

\newcommand\R{\mathbb{R}}

\begin{document}

\articletype{ARTICLE TEMPLATE}

\title{A robust regularized extreme learning machine for regression problems based on self-adaptive accelerated extra-gradient algorithm}

\author{
\name{ Muideen Adegoke\textsuperscript{a}, Lateef O. Jolaoso\textsuperscript{b}\thanks{CONTACT L.O.J. Author. Email: l.o.jolaoso@soton.ac.uk} and Mardiyyah Oduwole\textsuperscript{c}} 
\affil{ \textsuperscript{a}Big Data Technologies and Innovation Laboratory, University of Hertfordshire,  Hatfield AL10 9AB, United Kingdom; \textsuperscript{b}School of Mathematical Sciences,	University of Southampton, SO17 1BJ  United Kingdom; Department of Mathematics and Applied Mathematics, Sefako Makgatho Health Sciences University, P.O. Box 94 Medunsa 0204, Pretoria, South Africa; \textsuperscript{c}ML Collective; Department of Computer Science, National Open University Nigeria, Nigeria.}}

\maketitle

\begin{abstract}
The Extreme Learning Machine (ELM) technique is a machine learning approach for constructing feed-forward neural networks with a single hidden layer and their models. The ELM model can be constructed while being trained by concurrently reducing both the modeling errors and the norm of the output weights. Usually, the squared loss is widely utilized in the objective function of ELM, which can be treated as a LASSO problem and hence applying the fast iterative shrinkage thresholding algorithm (FISTA). However, in this paper, the minimization problem is solved from a variational inequalities perspective giving rise to improved algorithms that are more efficient than the FISTA. A fast general extra-gradient algorithm which is a form of first-order algorithm is developed with inertial acceleration techniques and variable stepsize which is updated at every iteration. The strong convergence and linear rate of convergence of the proposed algorithm are established under mild conditions. In terms of experiments, two experiments were presented to illustrate the computational advantage of the proposed method to other algorithms.
\end{abstract}


\section{Introduction}\label{Sec1:Introduction}
The extreme learning machine (ELM)~\cite{huang2006extreme} concept is a widely used machine learning method for constructing single-hidden layer feed-forward neural networks (SLFNNs). ELM networks and models are employed in many applications, namely control systems, image denoising, signal processing, classification problems, and regression problems~\cite{ding2014extreme}.  The concept explored in the  ELM techniques involves feeding-forward input data to the network's output. The ELM network transforms the input data into new features in another space using randomly generated weights and biases. In the ELM concept, an activation function then processes these new features, allowing for the extraction of non-linear features which can then be used by Output weight for making prediction

As aforementioned, in ~\cite{huang2006extreme}  extreme learning machine (ELM)  is proposed as a solution for constructing single-hidden layer feedforward networks (SLFNs). It has the capacity to approximate nonlinear functions of input data of various data representations. Unlike traditional methods, ELM randomly generates the parameters of hidden nodes, the weight, and the bias between the input layer and the hidden layer, and once the parameters are generated, they are left untunned throughout the training of the ELM network/model and only the output weights are learned. It has been proven that ELM exhibits the universal approximation capability~\cite{huang2006universal}. ELM offers several merits, these include easy and fast implementation, fast training speed, and generalization performance. As a result, ELM has gained significant interest in various regression problems, such as stock market forecasting \cite{zhao2012online}, price forecasting \cite{chen2012electricity}, wind power forecasting \cite{wan2013probabilistic}, and affective analogical reasoning \cite{cambria2015elm}. 

Despite the popularity of the classical ELM network/model,  the training of the classical ELM  can easily overfit if proper care is not taken. To address this issue, several versions of classical algorithms to train ELM networks were proposed. For instance,  Deng et al. \cite{deng2009regularized} proposed a regularized ELM with weighted least squares to enhance robustness. Their algorithm consists of two stages of reweighted ELM. Zhang et al. \cite{zhang2015outlier} introduced an outlier-robust ELM utilizing the $l_1$-norm loss function and the $l_2$-norm regularization term. They employed the augmented Lagrange multiplier algorithm to effectively reduce the influence of outliers. Horata et al. \cite{horata2013robust} adopted the Huber function to enhance the robustness and utilized the iteratively reweighted least squares (IRLS) algorithm to solve the Huber loss function without a regularization term. However, models without regularization are prone to overfitting.

Furthermore, despite the advancements made by existing robust ELM regression methods, namely, utilizing $l_1$-norm or Huber function, they are still susceptible to outliers with large deviations due to the linearity of these loss functions with respect to deviations. Besides, existing methods employ only $l_2$-norm regularization or lack a regularization term altogether. When the number of hidden nodes is large, the $l_2$-norm regularization tends to train a large ELM model due to the non-zero output weights of the network. 

Also, it should be noted that in the Extreme Learning Machine (ELM) approach, the output weights can be determined analytically using a Moore-Penrose generalized inverse. Conversely, the input weights and biases are randomly generated. However, if the output matrix is problematic, it can result in issues such as overfitting, inaccurate estimation, and unstable behavior. To overcome these challenges, a regularization term and the squared loss of prediction errors are often incorporated, forming a regularized objective function. This approach effectively addresses the problems associated with overfitting and ill-posed situations, thereby enhancing the predictive capability of ELM. Commonly employed regularization terms include the L2 norm and L1 norm-based terms. Generally, the latter proves more advantageous when dealing with ill-posed problems since it tends to produce sparser and more resilient output weights.

Therefore, this study addresses this gap by focusing on the robust ELM regression problem with the $l_2$ loss function and $l_1$regularization term. Considering the challenges posed by outliers and the need for appropriate regularization techniques, we formulated the problem as a variational inequality problem and proposed a new learning algorithm that can enhance the performance of the ELM network/model under outlier situations. In the paper, we consider non-linear regression problems as case studies to validate the performance of the proposed algorithm and approach. Further details are given in the remaining body of this article.

The remaining sections of the papers are as follows, Section 1 presents an introduction and general background detail of the work. Section 2 describes the ELM concept and model development. Furthermore, presents variational inequality and formulation. Section 3, presents the proposed algorithm. Besides, the convergence results and analysis of the proposed algorithm are detailed in this section. Similarly, Section 4 details numerical experiments which include comparisons of algorithms, and details of data use to mention a few. The conclusion is presented in Section 5.

\section{Extreme learning machine (ELM)}
The classical Extreme learning machine network is a three-layer feed-forward neural network (FFNN). It is widely used for many applications ranging from regression tasks, classification tasks,  feature extraction, and dimensional reduction among others. Figure \ref{ELM_figure} shows the structure of a single-hidden layer feed-forward neural network. It consists of an input layer, a hidden layer, and the output layer. Also, between the input layer and the hidden layer are the weight and bias denoted by $\mathbf{W}$ and $\mathbf{b}$ respectively. Similarly, between the hidden layer and the output layer is the output weight $\mathbf{\beta}$. The input layer takes the input data set represented by $\mathbf{X}$ and transforms it using the randomly generated weight and bias. The transformed input is given by
\begin{eqnarray}
    \bar{\mathbf{X}} = \mathbf{X}\mathbf{W} + \mathbf{b}.
\end{eqnarray}
The transformed input $\bar{\mathbf{X}}$ is passed to the hidden layer which contains an activation function such as the Sigmoid function which is given by
\begin{eqnarray}
    \sigma(x) = \frac{1}{1 + \exp^{-x}}.
\end{eqnarray}
Note that the activation function introduces non-linearity to its input. Thus, the output of the hidden layer  is given by

\begin{eqnarray}
    \mathbf{H} = \sigma(\bar{\mathbf{X}}).
\end{eqnarray}
\begin{figure}
    \centering
    \includegraphics[width=\columnwidth]{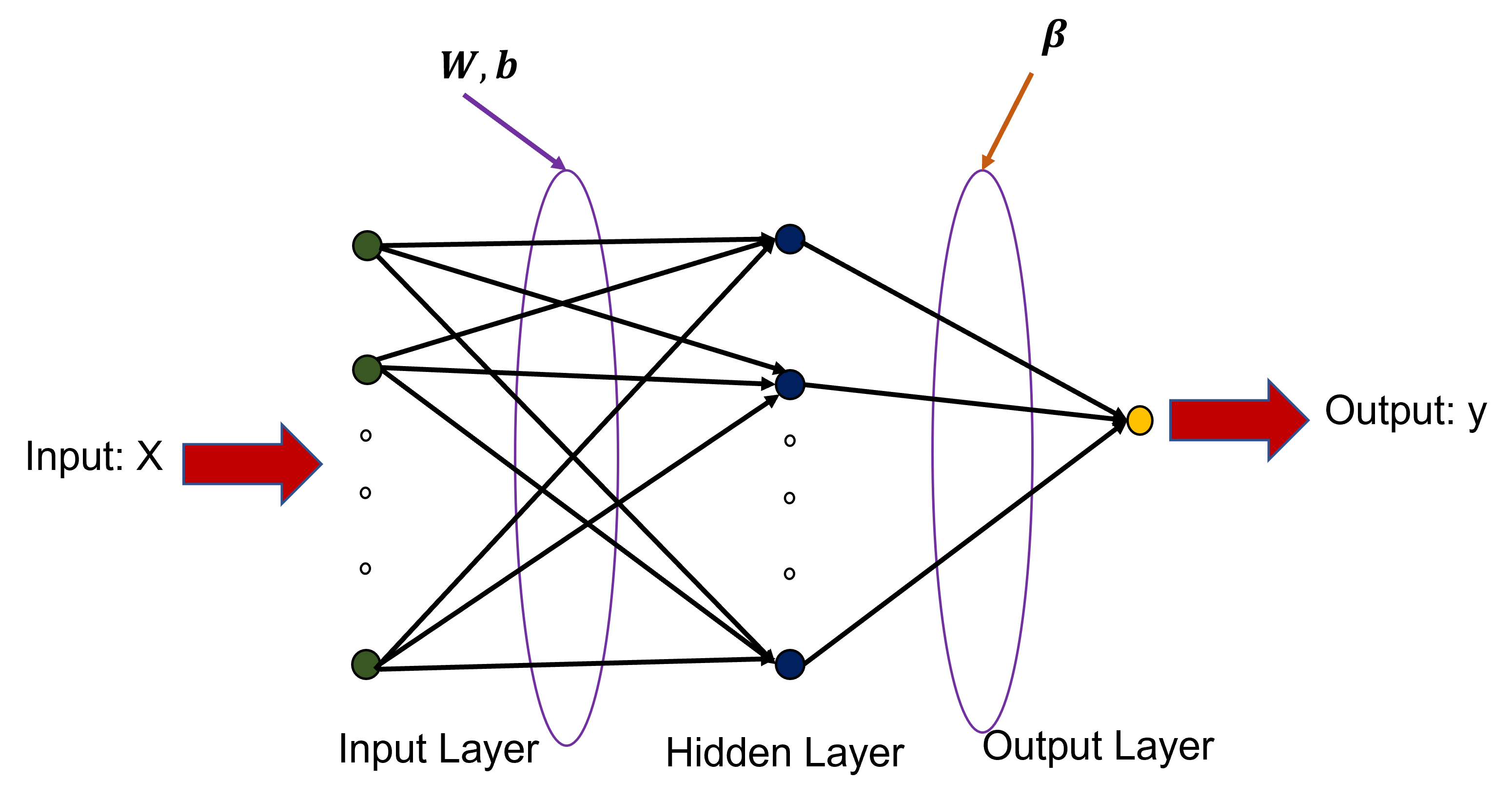}
    \caption{The figure of a Single Layer based ELM}
    \label{ELM_figure}
\end{figure}

\begin{figure}
    \centering
    \caption{Caption}
    \label{fig:enter-label}
\end{figure}

\noindent The output layer computes the output of the network using the output of the hidden layer and the output weight. The output of the network $\mathbf{o}$ is given by
\begin{eqnarray}
    \mathbf{o} = \mathbf{H}\mathbf{\beta}.
\end{eqnarray}
An important step in any machine learning problem is the training of the model. In the classical ELM concept, the output weight is obtained using the least square method. In other words, in the traditional ELM,  to obtain optimal output weight $\beta^*$, the least square method was employed. In the next subsection, we dive deep into the details of the problem this paper considered and the classical way to obtain optimal output weight $\beta^*$.


\subsection{Prediction problem to be considered}
In this paper, we consider a non-linear regression problem. Given training data sets denoted as $\mathbb{X}_{train}=\{(\boldsymbol{x_n},y_n): \, \boldsymbol{x}_n \, \in \, \mathbb{R}^D, y_n \, \in \mathbb{R}, \quad n=1,\ldots,N \}$, Here, $D$ represents the number of input features, $N$ represents the number of training samples, and $\boldsymbol{x_n}$ and $y_n$ represent the inputs and target outputs, respectively, for the $k$th sample. Similarly, the test set is denoted as follows $\mathbb{X}_{test}=\{\boldsymbol{x}'_{n'},y'_{n'}):\, \boldsymbol{x}'_{n'} \, \in \, \mathbb{R}^{D'}, y'_{n'} \in \mathbb{R}, \quad n'=1,\ldots,N' \}$. In this case, $N'$ represents the number of samples in the test set. When considering a neural network with $m$ hidden nodes, the output of a Single Layer Feed-forward Network (SLFN) can be expressed as follows
\begin{equation}
\label{eq1}
g_m(\boldsymbol{x}_n)=\sum_{j=1}^m \beta_j h_j(\boldsymbol{x}_n),
\end{equation}
where $\beta_j$ is the $j$th output weight and $h_j(\cdot)$ is the output of the $j$th hidden node or activation function. By grouping all the weights and hidden nodes, then (\ref{eq1}) can be written in a compact form as
\begin{equation}
\label{eq2}
g_m(\boldsymbol{x}_n)=\boldsymbol{h}_m(\boldsymbol{x}_n)\boldsymbol{\beta},
\end{equation}
where $\boldsymbol{\beta}=\{ \beta_1,\beta_2,\cdots,\beta_m \}$, $\boldsymbol{h}_m(\boldsymbol{x}_n)=\{h_1(\boldsymbol{x}_n),h_2(\boldsymbol{x}_n),\cdots,h_m(\boldsymbol{x}_n)\}$. There are several activation functions, for instance, in the case of the sigmoid activation function, $h_j(\cdot)$ is given by
\begin{equation}\label{eq3}
h_{j}(\boldsymbol{x}_n)=\frac{1}{1+\exp^{-(\boldsymbol{a}_j^T\boldsymbol{x}_n + b_j)}}=\frac{1}{1+\exp^{(-\boldsymbol{w}_j^T\boldsymbol{o})}},
\end{equation}
where $\boldsymbol{w}_j=[\boldsymbol{a}_j^T,b_j]^T$ denotes grouped input weight~($\boldsymbol{a}_j$), and bias~($b_j$), at the $j$th hidden node. $\boldsymbol{o}=[\boldsymbol{x}_n^T,1]^T$ representing input sample. Given all training data sets, the output of the activation function is given by

\begin{equation}\label{eq4}
  \boldsymbol{H}=  \left(
                               \begin{array}{ccc}
                                 h_1(\boldsymbol{x}_1) & \cdots & h_m(\boldsymbol{x}_1) \\
                                 \vdots              & \ddots                & \vdots \\
                                 h_1(\boldsymbol{x}_N) & \cdots & h_m(\boldsymbol{x}_N) \\
                               \end{array}
                    \right).
\end{equation}
For notation simplification, throughout the rest of the discussion, we would refer to $g_m(\boldsymbol{x}_n)$ as $g_m$. Also, $\boldsymbol{h}_m(\boldsymbol{x}_n)$, $h_j(\boldsymbol{x}_n)$ as $\boldsymbol{h}_m$ and $h_j$ respectively.
The training set error (the difference between the network output and training target) is given by
\begin{equation}
\label{eq5}
    \xi = \sum_{n=1}^N (y_n-g_m\beta_m)^2 = \left \| \boldsymbol{y}- \boldsymbol{H}\boldsymbol{\beta} \right \|_2^2
\end{equation}
where $\boldsymbol{y}=\{ y_1,y_2,\cdots,y_N \}$. As aforementioned, ELM uses the least square approach to obtain its  output weight, hence, the optimal output weight that minimizes the training set error $\xi$ is given by
\begin{equation}\label{eq6}
\boldsymbol{\beta}^*= \left ( \boldsymbol{H}^T \boldsymbol{H} \right)^{-1}\boldsymbol{H}^T \boldsymbol{T}.
\end{equation}
Nevertheless, achieving a sparse optimal output weight requires a reformulation of the objective function \ref{eq5} due to the dense nature of the optimal weight \ref{eq6}. This reformulation aims to effectively handle the sparse weight provided by the new objective function and it is given by
\begin{eqnarray} \label{eq11}
   J_{elm} = \underset{\boldsymbol{\beta}}{\min} \;  \| \boldsymbol{y}- \mathbf{H}\boldsymbol{\beta}\|_{2}^2 + \lambda\|\boldsymbol{\beta}\|_1.
\end{eqnarray}
where $\lambda$ is a regularization parameter. The ELM task thus becomes solving the optimization problem \eqref{eq11} to obtain the output weight $\beta$. In the next subsection, we approach solving the optimization problem \eqref{eq11} via the variational inequalities technique which is a general framework and classical method in the literature.

\subsection{A variational inequalities approach}
{In this subsection, we lay some foundation for the proposed learning algorithm. The algorithm is then developed and employed to train an extreme learning machine network in the further section. The algorithm is developed based on a variational inequalities model formulated via standard convex optimization. }

For a given hidden layer output matrix $\textbf{H}\in \R^{N \times m}$, output weight $\boldsymbol{\beta} \in \R^m$ and target output $y \in \R^N,$ the model \eqref{eq11} can be reformulated as optimization problem of the form 
\begin{align}\label{min}
  & \underset{\beta}{{\min}}\; \mathcal{E}(\boldsymbol{\beta}):= \iota_{\mathcal{K}}(\beta)+ \|\boldsymbol{y} - \mathbf{H}\boldsymbol{\beta}\|_{2}^2, 
\end{align}
where $\iota_{\mathcal{K}}$ is the indicator function on the set $\mathcal{K}$ defined by $\{ \boldsymbol{\beta} \in \mathbb{R}^N \; | \; \|\boldsymbol{\beta}\|_{1} \leq 1 \}.$
Denote by $\boldsymbol{\beta}^* := \inf_{\boldsymbol{\beta} \in \mathcal{K}}\mathcal{E}(\boldsymbol{\beta}),$ the optimal objective value of \eqref{min}. Notice that $\mathcal{E}$ is a convex objective function with continuous gradient $\nabla \mathcal{E}(\boldsymbol{\beta}) = \textbf{H}^\top(\textbf{H}\boldsymbol{\beta} - \boldsymbol{y}),$ which is Lipschitz continuous. In order to solve the optimization from \eqref{min}, we adopt a generic framework using the concept of variational inequalities. The variational inequalities is a classical and general framework that encompasses a wide variety of optimization problems such as convex minimization, equilibrium problems, convex-concave saddle point problems, and inclusion problems, which are ubiquitous in machine learning and optimization ( \cite{Nemi2004, Judi2011, Judi&Nemi2016}). Given a convex subset $\mathcal{K}$ of $\mathbb{R}^N,$ these inequalities are  often desgined from a monotone operator $F:\mathcal{K} \to \mathbb{R}^N$ such that for any $(x,y) \in \mathcal{K} \times \mathcal{K},$ $(F(x)- F(y))^\top \cdot (x-y) \geq 0.$ The goal is then to find a solution $x^\dagger \in \mathcal{K}$ to the variational inequalities, that is, such that 
\begin{equation}
    \forall x \in \mathcal{K}, \quad F(x^\dagger)^\top\cdot (x - x^\dagger) \leq 0. \label{vip}
\end{equation} For the case of the reformulated problem \eqref{min}, the operator $F$ is simply the gradient of $\mathcal{E}.$ While our motivation is to solve the convex minimization arising in \eqref{min}, the variational inequalities framework is more general (see, e.g., \cite{Nemi2004, Bach&Levy2019} and reference therein). In this paper, we are interested in the algorithm to solve the inequality in \eqref{vip}, while only accessing an oracle for $F(x)$ for any given $x \in \mathcal{K},$ or only an unbiased estimate of $F(x)$. We also assume that we may efficiently project onto the set $\mathcal{K}$ (which we will assume is nonempty, closed, and convex throughout this paper). In terms of complexity bounds, this problem is by now well understood with matching upper and lower bounds in a variety of situations. In particular, the notion of smoothness (i.e., Lipschitz-continuity of $F$ vs. simply assuming that $F$ is bounded) and the presence of noise are two important factors influencing the convergence rate. For example, in the projection gradient algorithm of \cite{fac,guler, gold}, the convergence of the method depends heavily on the parameter of strong monotonicity of $F$ which in general yields a weak convergence iterate. Also in the extragradient methods of \cite{cens1,cens3,korp}, the algorithms depend heavily on a prior estimate of the Lipschitz constant of $F$ while requiring two projections onto the set $\mathcal{K}$. Both requirements make the algorithms not to be best suitable for solving variational inequalities in complex cases. Other modifications of the above-listed methods can be found in \cite{cens4, Malik, sol, yang} and references therein. Let us mention in particular, the method due to Tseng \cite{tseng} which requires a single projection onto the set $\mathcal{K}$ per each iteration. This method exploits the value of $F$ at the previous and current iterations of the algorithm. Although the original method depends on the prior estimate of the Lipschitz constant of $F$, other improved versions of the method have been introduced in, for instance, \cite{Sun, Thong1, Thong2, CaiYekini}. 

The choice of an adequate stepsize for an algorithm for solving the variational inequalities has attracted effort from many researchers in recent years. Apart from the stepsize depending on the Lipschitz constant of $F$ as mentioned earlier, an initial guess can be too small or too large which slows down the convergence rate of the algorithm. In order to overcome this challenge, a universal approach which is determined via a self-adaptive process has been considered in many instances, see, e.g. \cite{Jol1, Jol2, Jol3}. This universal approach is bounded by a lower and an upper bound which is updated at every iteration. Furthermore, the choice of inertial term which originated from heavy ball discretization of a second-order dynamical system has been considered as an acceleration technique for several algorithms \cite{BeckFISTA, Moudafi, Cham, Polyak}. This approach updates the next iterate by using a memory of the previous two iterates of the algorithm, see, e.g. \cite{Jol1, Jol2, Jol3, Sun}. In this paper, we make the following contributions in terms of the proposed algorithm:
\begin{itemize}
    \item We present a general adaptive accelerated method for variational inequalities based on the extragradient method in \cite{tseng}. Our method employs a simple self-adaptive choice of stepsize that leads to an improved rate of convergence for the algorithm. Moreover, our algorithm does not require prior knowledge regarding the Lipschitz constant of $F$.
    \item We accelerate the performance of the algorithm by incorporating double inertial extrapolation steps which helps to speed up the convergence rate of the algorithm.
    \item We also include a relaxation parameter that tunes the iteration of the algorithm towards the solution of the variational inequalities.
    \item On the technical side, we provide the convergence analysis and linear convergence rate of the proposed algorithm.
\end{itemize}

\subsection{Mathematical tools}
In this subsection, we recall some mathematical tools and notations needed in this paper. We denote by $\|\cdot\|$, the induced norm on $\R^n$ and $\langle \cdot,\cdot \rangle$, the inner product on $\R^n \times \R^n.$ The weak convergence of a sequence $\{x_n\}$ to $x$ as $n \to \infty$ is denoted by $x_n \rightharpoonup x$ and the strong convergence of $\{x_n\}$ to $x$ as $n \to \infty$ by $x_n \to x.$ More so, $\omega(x_n)$ denotes the set of weak accumulation point $x_n.$ For each $x,y \in \R^n,$ it is clear that
\begin{itemize}
    \item $\|x+y \|^2 \leq \|x\|^2 + 2 \langle y,x+y \rangle,$
    \item $\|x-y\|^2 = \|x\|^2 - 2\langle x,y \rangle + \|y\|^2.$
\end{itemize}
Also for any $t \in [0,1],$ we have
\begin{equation}
    \|tx + (1-t)y\|^2 = t\|x\|^2 + (1-t)\|y\|^2 - t(1-t)\|x-y\|^2. 
\end{equation}
For a nonempty, closed and convex subset $\mathcal{K} \subset \R^n,$ the nearest point of $x \in \R^n$ to $\mathcal{K}$ (called projection onto $\mathcal{K}$) is the unique vector $P_{\mathcal{K}}x$ satisfying
\begin{equation}
    \|x - P_{\mathcal{K}}x\| \leq \|x - y \| \quad \forall y \in \mathcal{K}.
\end{equation}
It is well known that $P_{\mathcal{K}}$ is Lipschitz continuous with Lipschitz constant 1 and satisfies the following identity:
\begin{equation}
    u  = P_{\mathcal{K}}x  \Leftrightarrow \langle x - u, u - w \rangle \geq 0\quad \forall w \in \mathcal{K}. \label{proj}
\end{equation}
In addition, $P_{\mathcal{K}}$ is characterized by the following inequality:
\begin{equation}
    \|P_{\mathcal{K}}x - u\|^2 \leq \|x - u \|^2 - \| x - P_{\mathcal{K}}x\|^2 \quad u \in \mathcal{K}.
\end{equation}
 
\begin{lemma}\cite{Att}
       Let $\{\phi_{n}\}, \{\alpha_{n}\}$ and $\{\psi_{n}\}$ be sequences of nonnegative real number satisfying 
       \begin{equation}
    \phi_{n+1} \leq \phi_{n} + \alpha_n (\phi_n - \phi_{n-1}) + \psi_n, \quad \forall n \geq 1, \; \sum_{n=1}^{+\infty}\psi_n < + \infty,
\end{equation}
and  such that $0 \leq \alpha_{n} \leq \alpha < 1$ for all $ n \geq 1.$ Then the limit $\lim_{n\to + \infty} \phi_{n} \in \R$ exists.
   \end{lemma}

\begin{lemma}\label{viplemma}\cite{Taka}
    Let $\mathcal{K}$ be a nonempty, closed and convex subset of $\mathbb{R}^N.$ Suppose that $F: \mathcal{K} \to \R^n$ is continuous and monotone. Then for any $p \in \mathcal{K},$ we have
    \begin{equation}
        p \in VI(F,\mathcal{K}) \Leftrightarrow \langle Fx, x - p \rangle \geq 0, \quad \forall x \in \mathcal{K}.
    \end{equation}
\end{lemma}

We recall two notions of linear convergence, which will be used in our convergence analysis. For a sequence $\{s_n\}$, we say that $\{s_n\}$ converges Q-linearly to $s^*$ if there exist $\alpha \in (0,1)$, $\rho >0$ and $n_0>0$ such that
\begin{equation*}
    \|s_{n+1} - s^*\| \leq \alpha \|s_n - s^* \| \quad \forall  n \geq n_0,
\end{equation*}
and we say that $\{s_n\}$ converges R-linearly to $s^*$ if 
\begin{equation*}
    \|s_n - s^*\| \leq \rho \alpha^n \quad  \forall n \geq n_0.
\end{equation*}

\section{Proposed algorithm and convergence results}
In this section, we propose our algorithm and investigate the convergence analysis of the proposed method. In the sequel, we assume that the following assumptions hold on the cost operator $F$.

\begin{assumption}\label{as1}
    The operator $F: \R^n \to \R^n$ is monotone and Lipschitz continuous, i.e.,  
    \begin{equation*}
        \langle Fx - Fy, x - y \rangle \geq 0 \qquad \text{and} \quad \|Fx - Fy\| \leq L\|x-y\| \quad \forall x,y \in \R^n \; \text{and} \; L>0.
    \end{equation*}
\end{assumption}
\begin{assumption}\label{as2}
    The solution set $VI(F,\mathcal{K})$ is nonempty.
\end{assumption}
\noindent Note that Assumption \ref{as1} is satisfied by the cost operator $\nabla \mathcal{E}(\cdot) = \mathbf{H}^\top (\mathbf{H}(\cdot)-y)$ from problem \eqref{min}. In particular $\nabla \mathcal{E}$ is $\|\mathbf{H}^\top \mathbf{H}\|$-Lipschitz continuous. More so, Assumption \ref{as2} is typical for solving variational inequalities. Now we present our algorithm as follows.

\subsection{Proposed algorithm}
\begin{algorithm} \label{game}
  \caption{A general adaptive accelerated method (GAME)}
  \begin{algorithmic}
  \STATE {\bf Inputs:} Given starting points $x_{0},x_{-1} \in \R^n$, sequence $\{\alpha_n\},\{\beta_n\} \subset (0,1),$ a relaxation parameter $\rho \in [0,1],$ $\mu \in (0,1),$ $\lambda_0>0$ and $\{\zeta_n\}\subset (0,\infty).$ 
  \STATE {\bf Main iterate:} Set $n = 1,2, \dots,$ do: \\
  \texttt{Step 1:} Compute
  \begin{align*}
      & a_n = s_n + \alpha_n (s_n - s_{n-1})\\
      & b_n = s_n + \beta_n (s_n - s_{n-1})\\
      & c_n = P_{\mathcal{K}}(b_n - \lambda_n F(b_n))\\
      & s_{n+1} = (1-\rho)a_n + \rho (c_n - \lambda_n (F(c_n) - F(b_n))).
  \end{align*}
  \texttt{Step 2:} Update the stepsize via
  \begin{equation*}
      \lambda_{n+1} =  \begin{cases}
          \min \left\{ \frac{\mu \|b_n - c_n \|}{\|F(b_n) - F(c_n)\|}, \lambda_{n} + \zeta_n  \right\} \qquad & \text{if} \; F(b_n) \neq F(c_n),\\
          \lambda_n + \zeta_n, & \text{otherwise}.
      \end{cases}
  \end{equation*}
  \texttt{Step 3:} If a stopping criterion is not met, then set $n := n+1$ and goto \texttt{Step 1}.\\
  \textbf{Output:} $s_{n+1} \in \mathcal{K},$ which is the approximate weight $\beta$ of the ELM network in \eqref{eq11}.
  \end{algorithmic}
\end{algorithm}


Before we start the convergence analysis, we first check for a possible stopping criterion which will be used in Step 3 of the proposed method. In particular, we used $\|b_n - c_n \|=0$ as a possible stopping condition which guarantees that the sequence $\{s_n\}$ generated by Algorithm \ref{game} converges to a solution of the variational inequalities.


\begin{lemma}\label{Res1}
    Let $c_n=b_n$ for some $n \geq 0$ in Algorithm \ref{game} happened, then $b_n$ is a solution of the variational inequality problem, i.e., $b_n \in VI(F,\mathcal{K}).$
\end{lemma}
\begin{proof}
    From the characterization inequality of the projection mapping in \eqref{proj} (noting that $c_n =b_n = \mathcal{P}_{\mathcal{K}}(b_n - \lambda_n F(b_n)) $), then we have
    \begin{equation*}
        \langle b_n - u, b_n - \lambda_n F(b_n) - b_n \rangle \geq 0, \qquad \forall u \in \mathcal{K}. 
    \end{equation*}
    Hence $$ \lambda_n \langle F(b_n),u-b_n \rangle \geq 0, \qquad \forall u \in \mathcal{K}.$$
    Let $\lim_{n \to \infty}\lambda_n = \lambda>0,$ then we obtain $\langle F(b_n), u - b_n\rangle \geq 0$ for all $u \in \mathcal{K}.$ This implies that $b_n \in VI(F,\mathcal{K}).$
\end{proof}

In the proof of Lemma \ref{Res1}, we used the fact that $\lim_{n\to\infty}\lambda_n = \lambda >0.$ We justify this fact in the next result and show that indeed, this limit exists and $\{\lambda_n\}$ is nonincreasing.
\begin{lemma}
    The sequence $\{\lambda_n\}$ generated in Step 2 of Algorithm \ref{game} is nonincreasing and $\lim_{n\to \infty}\lambda_n = \lambda \in \left[ \min \left\{ \lambda_0, \frac{\mu}{L}\right\}, \lambda_0 + \zeta_{0} \right]$, where $L$ is the Lipschitz constant of $F.$
\end{lemma}

\begin{proof}
    It is clear from the definition of $\{\lambda_n\}$ that $\lambda_{n+1} \leq \lambda_n$ for all $n \geq 0.$ Hence, $\lambda_n$ is nonincreasing. From the Lipschitz continuity of $F$, we obtain
    \begin{equation*}
        \frac{\mu \|b_n - c_n\|}{\|F(b_n)-F(c_n)\|} \geq \frac{\mu}{L}, \quad \text{if} \; F(b_n) \neq F(c_n).
    \end{equation*}
    Hence, $\lambda_{n+1}\geq \min \left\{\lambda_0, \frac{\mu}{L}\right\}$ for all $n \geq0.$ This implies that the limit exists and $\lim_{n\to\infty}\lambda_n = \lambda \in \left[ \min \left\{ \lambda_0, \frac{\mu}{L}\right\}, \lambda_0 + \zeta_{0} \right].$
\end{proof}

\begin{remark} \label{Rmk1}
Now we justify the term `\textit{general}' used in naming Algorithm \ref{game}. The following is the relation of the proposed algorithm with other related works in the literature.
\begin{itemize}
    \item[-] Note that if we remove the relaxation parameter in the update step by setting $\rho_n = 1,$ we obtain the double inertial extragradient method (DIEM) which is related to the algorithms proposed in \cite{Sun}.
    \item[-] Also if one of the inertial terms is removed in Algorithm \ref{game} by setting $\alpha_n = 0,$ we obtain an inertial relaxed extragradient method (IREM) which is related to the algorithms proposed in \cite{Bot, Thong1, CaiYekini, Thong2}.
    \item[-] If both inertial terms are removed in Algorithm \ref{game}, we obtain a relaxed extragradient method (REM) which is related to the methods studied in \cite{Bot2, Alak}
    \item[-] Finally, if both the double inertial terms and the relaxation parameter are removed from Algorithm \ref{game}, we obtain the extragradient method (EM) studied in \cite{tseng,wang}.
\end{itemize}
Note that despite the relationship of the proposed Algorithm \ref{game} with the above literature, the proposed algorithm uses a self-adaptive technique for selecting the stepsize which is updated at every iteration. This approach is more simple and efficient for solving the variational inequality problem than most of the methods in the literature.
\end{remark}

\subsection{Convergence results}
We shall discuss the convergence behavior of Algorithm \ref{game}. First, we note that it is immediate from the definition of $s_{n}$-update that
\begin{equation}\label{sn}
    s_{n+1} = (1-\rho)a_n + \rho e_n \quad \text{and} \quad \frac{1}{\rho}(s_{n+1} -a_n) = e_n - a_n
\end{equation}
 where $e_n = c_n - \lambda_n(Fc_n - Fb_n).$ Let $s^*$ be a solution of the variational inequalities, i.e., $s^* \in VI(F,\mathcal{K}).$ For $n \geq 0,$ we denote
\begin{align*}
   & H_n = \|s_n -s^*\|^2 - \left((1-\rho)\alpha_n + \rho \beta_n\right)\|s_{n-1}-s^*\|^2 + \frac{1-\rho}{\rho}(1-\alpha_n)\|s_n-s_{n-1}\|^2.
   \end{align*}
We will study the convergence properties of $\{H_n\}$ in the next subsection. The results will then be used to establish the convergence of $\{s_n\}.$  The following result can easily be proved from available literature, e.g. \cite{LiuYang}.  
\begin{lemma}\label{lem2}
    Let $\{s_n\}$ be the sequence generated by Algorithm \ref{game} and $s^* \in VI(F,\mathcal{K}).$ Then for $n \geq 0,$
    \begin{equation*}
        \|e_n -s^*\|^2 \leq \|b_n -s^*\|^2 - \left( 1 - \frac{\mu^2\lambda_n^2}{\lambda_{n+1}^2} \right)\|b_n-c_n\|^2.
    \end{equation*}
\end{lemma}

We start by showing that $\{H_n\}$ is non-increasing and convergent.
\begin{lemma}
    Let $\{s_n\}$ be the sequence generated by Algorithm \ref{game}. Then the following statements hold:
    \begin{itemize}
        \item[(i)] the sequence $\{H_n\}$ is monotonically non-increasing and in particular, for $n \geq 0,$ $$H_{n+1}-H_{n} \leq - \left( \frac{1 - \rho}{\rho}(1 - \alpha_n)^2 + \rho (\alpha_n ( 1+ \alpha_n) - \beta_n(1+\beta_n)) - (1+\alpha_n)\alpha_n \right)\|s_n-s_{n-1}\|^2.$$
        \item[(ii)] $\sum_{n=0}^\infty \|s_n - s_{n-1}\|< \infty$.
    \end{itemize}
\end{lemma}

 \begin{proof}
 We first prove (i). Let $s^* \in VI(F,\mathcal{K}).$ Then, from the definition of $s_{n+1},$ we have
    \begin{eqnarray}
        \|s_{n+1} - s^*\|^2 &\leq& (1 - \rho) \|a_n - s^* \|^2 + \rho \|b_n - s^*\|^2 - \rho(1- \rho) \|a_n  -e_n\|^2. \nonumber
    \end{eqnarray}
Also from \eqref{sn}, we get
    \begin{equation}
        \|s_{n+1} - s^* \|^2 \leq (1- \rho) \|a_n - s^*\|^2 + \rho \|b_n - s^*\|^2 - \frac{(1 - \rho)}{\rho}\|s_{n+1} - a_n\|^2. \label{3.1}
    \end{equation}
    Now let us estimate each of the terms on the LHS of \eqref{3.1}.  Clearly
    \begin{eqnarray}
        \|a_n - s^*\|^2 &=& \| s_n + \alpha_n (s_n - s_{n-1}) - s^* \|^2 \nonumber\\
        &=& \|(1 + \alpha_n)(s_n - s^*) - \alpha_n (s_{n-1}-s^*) \|^2 \nonumber\\
        &=& (1+ \alpha_n)\|s_n - s^*\|^2 - \alpha_n \|s_{n-1}-s^*\|^2 + \alpha_n(1+\alpha_n)\|s_n - s_{n-1}\|^2. \nonumber
    \end{eqnarray}
    Similarly
    \begin{equation}
        \|b_n - s^*\|^2 = (1+ \beta_n)\|s_n - s^*\|^2 - \beta_n\|s_{n-1}-s^*\|^2 + \beta_n(1+\beta_n)\|s_n - s_{n-1}\|^2. \nonumber
    \end{equation}
    Also
    \begin{eqnarray}
        \|s_{n+1} -a_n \|^2 &=& \|s_{n+1} - s_n \|^2 + \alpha_n^2 \|s_n - s_{n-1}\|^2 - 2 \alpha_n \langle s_{n+1} - s_{n}, s_n - s_{n-1} \rangle \nonumber\\
       & \geq& \|s_{n+1} - s_n \|^2 + \alpha_n^2 \|s_n - s_{n-1}\|^2 - \alpha_n ( \|s_{n+1} - s_n \|^2 + \|s_n - s_{n-1}\|^2 ) \nonumber\\
        &=& (1 - \alpha_n)\|s_{n+1} - s_{n}\|^2 + (\alpha_n^2 - \alpha_n) \|s_n  -s_{n-1}\|^2. \nonumber
    \end{eqnarray}
    Substituting the above identities in \eqref{3.1}, we get
    \begin{align}
        &\|s_{n+1} - s^*\|^2 \nonumber\\
        &\leq  (1-\rho) \left( (1+ \alpha_n)\|s_n - s^*\|^2 - \alpha_n \|s_{n-1} - s^* \|^2 + \alpha_n (1 + \alpha_n) \|s_n - s_{n-1}\|^2  \right) \nonumber\\
        & \; + \rho \left( (1+\beta_n)\|s_n - s^*\|^2 - \beta_n \|s_{n-1} - s^*\|^2 + \beta_n(1+\beta_n)\|s_n - s_{n-1}\|^2 \right) \nonumber\\
        &\; - \frac{1-\rho}{\rho} \left( (1-\alpha_n)\|s_{n+1} - s_n\|^2 + (\alpha_n^2 - \alpha_n)\|s_n - s_{n-1}\|^2 \right) \nonumber\\
        &= \left( 1 + \alpha_n - \rho(\alpha_n - \beta_n) \right)\|s_n - s^*\|^2 - (\alpha_n - (\alpha_n-\beta_n)\rho) \|s_{n-1}-s^*\|^2 \nonumber\\
        &\; + \left( (1+\alpha_n)\alpha_n - \rho(\alpha_n(1+\alpha_n) - \beta_n(1+\beta_n)) - \frac{(1-\rho)}{\rho}(\alpha_n^2 - \alpha_n) \right)\|s_n - s_{n-1}\|^2 \nonumber\\
        &\; - \frac{(1-\rho)}{\rho}(1-\alpha_n)\|s_{n+1} - s_n\|^2.
        \label{3.2} 
    \end{align}
    Hence
    \begin{align}
         &\|s_{n+1}-s^*\|^2 - \left( \alpha_n - \rho (\alpha_n - \beta_n) \right) \|s_n - s^*\|^2 + \frac{1-\rho}{\rho}(1-\alpha_n)\|s_{n+1} - s_n\|^2 \nonumber\\
         &  \; \;\leq  \|s_n - s^*\|^2 - \left( \alpha_n - (\alpha_n - \beta_n)\rho \right) \|s_{n-1} - s^{*}\|^2 + \frac{1-\rho}{\rho}(1-\alpha_n)\|s_{n} - s_{n-1}\|^2  \nonumber\\
         & \; \;\; \; + \left( (1+ \alpha_n)\alpha_n - \rho(\alpha_n(1+\alpha_n) - \beta_n(1+\beta_n))  - \frac{1-\rho}{\rho}(\alpha_n - 1)^2 \right) \|s_n - s_{n-1}\|^2. \label{3.3}
    \end{align}
    Now choose $\rho$ and $\alpha_n$ such that 
    \begin{equation}
        0 < \rho < \frac{1}{1+\delta} \quad \text{and} \quad 0 \leq \beta_n \leq \alpha_n < 1 -  \sqrt{\frac{2}{\delta}}, \quad \quad \text{for some} \quad \delta >2. \nonumber
    \end{equation}
    Then
    \begin{align}
       & \left( (1 - \rho)\alpha_n(1+\alpha_n) + \rho \beta_n(1+\beta_n) - \frac{1-\rho}{\rho}(\alpha_n-1)^2 \right) \nonumber\\
       & \; \; \leq \left( 2\left((1 - \rho)\alpha_n + \rho \beta_n\right) - \frac{1-\rho}{\rho}(\alpha_n-1)^2 \right) \nonumber\\
       & \; \; \leq 2 - \frac{1-\rho}{\rho}(\alpha_n-1)^2  \leq 2 - \delta (\alpha_n-1)^2 \leq 0. \nonumber
    \end{align}
   It follows from \eqref{3.3} that
   \begin{equation}
       H_{n+1} - H_{n} \leq 0. \label{3.4}
   \end{equation}
   This implies that $H_n$ is monotonically non-increasing. The second conclusion follows from \eqref{3.3}. In addition, we get
  \begin{equation}\label{sum}
      H_{n+1} - H_n \leq - \theta_n \|s_n - s_{n-1}\|^2,
  \end{equation}
  where $\theta_n = \left( \frac{1 - \rho}{\rho}(1 - \alpha_n)^2 + \rho (\alpha_n ( 1+ \alpha_n) - \beta_n(1+\beta_n)) - (1+\alpha_n)\alpha_n \right).$ Note that $\liminf_{n \to \infty}\theta_n >0,$ hence, summing up \eqref{sum} from 1 to $N$, we see further that
  \begin{align}
      0 \leq \sum_{n=1}^{N} \theta_n \|s_n - s_{n-1}\|^2  \leq \sum_{n=1}^N \theta_n(H_n - H_{n+1}) = \theta_N(H_1 - H_{N+1}) < \infty.
  \end{align}
 This completes the proof. 
   \end{proof}

   The next result shows that any accumulation point of the sequence $\{s_n\}$ generated by Algorithm \ref{game} belongs to $VI(F,\mathcal{K}).$

\begin{theorem}\label{thmres}
Let $\{s_n\}$ be a sequence generated by Algorithm \ref{game} and $\|b_n - c_n\| \to 0$ as $n \to \infty.$ Then the sequence $\{s_n\}$ converges weakly to an element $\tilde{s} \in VI(F.\mathcal{K}).$
\end{theorem}
\begin{proof}
    Let $\tilde{s} \in \omega(s_n).$ Our aim is to show that $\tilde{s} \in VI(F,\mathcal{K}.$ Since $F$ is monotone, then $\langle Fb_n, x - b_n \rangle \leq \langle Fx, x- b_n \rangle.$ It follows from \eqref{proj} that for all $x \in \mathcal{K}$
    \begin{eqnarray}
        0 &\leq& \langle c_n - b_n - \lambda_n Fb_n, x - c_n \rangle \nonumber\\
        &=& \langle c_n - b_n, x - c_n \rangle + \lambda_n \langle Fb_n, b_n - c_n \rangle + \lambda_n \langle Fb_n, x - b_n \rangle \nonumber\\
        &\leq& \langle c_n - b_n, x - c_n \rangle + \lambda_n \langle Fb_n, b_n - c_n \rangle + \lambda_n \langle Fx, x - b_n \rangle. \nonumber
    \end{eqnarray}
    Passing limit to the above inequality, noting that $\lim_{n\to\infty}\lambda_n >0,$ then
    \begin{equation}
        \langle Fx, x - \tilde{s} \rangle \geq 0 \forall x \in \mathcal{K}.
    \end{equation}
    Hence, using Lemma \ref{viplemma}, we obtain that $\tilde{s} \in VI(F,\mathcal{K}).$
\end{proof}
In the prove of Theorem \ref{thmres}, we used the fact that $\lim \|b_n - c_n \|= 0.$ We shall establish this claim in the appendix.

Next, we show the convergence rate of Algorithm \ref{game} with the assumption that $F$ is $L$-Lipschitz continuous on $\R^n$ and $\kappa$-strongly pseudomonotone, i.e., there exists a constant $\kappa>0$ such that $\langle Fx,y-x \rangle \geq 0 \Rightarrow \langle Fy, y - x \rangle \geq \kappa\|x-y\|^2,$ for any $x,y \in \R^n.$ 

\subsection{Linear convergence result}
 
\begin{theorem}
    Let $\{s_n\}$ be the sequence generated by Algorithm \ref{game} and $F$ is $\kappa$-strongly pseudomonotone and Lipschitz continuous on $R^n$. Let $0 \leq \rho \leq \frac{1}{2}$ with $0 \leq \beta_n \leq \alpha_n \leq \min \left\{ \frac{1-2\rho}{1-\rho}, \frac{\rho}{1-\rho}, \frac{\rho\epsilon}{1-\rho\epsilon} \right\}$ where $\epsilon = \min \left\{ \frac{1-\mu^2}{2}, \frac{\mu\kappa}{L} \right\}.$ Then $\{s_n\}$ converges strongly to a solution $s^*$ of VIP with a R-linear rate.
\end{theorem}

\begin{proof}
Let $s^* \in VI(F,\mathcal{K}),$ then $\langle Fs^*, p - s^* \rangle \geq 0$ for all $p \in \mathcal{K}$. Hence $\langle Fp, p -s^* \rangle \geq \kappa \|p - s^*\|^2$ for all $p \in \mathcal{K}.$ Consequently
\begin{eqnarray}
    \langle Fc_n, s^* - e_n \rangle &=& \langle Fc_n, s^* - c_n \rangle + \langle Fc_n, c_n - e_n \rangle \nonumber\\
    &\leq& - \kappa \|c_n - s^* \|^2 + \langle Fc_n, c_n - e_n \rangle . \nonumber
\end{eqnarray}
From Lemma \ref{lem2}, we see that
\begin{equation}
    \|e_n -s^* \|^2 \leq \|b_n - s^* \|^2 - \left( 1 - \frac{\mu^2\lambda_n^2}{\lambda_{n+1}^2} \right)\|c_n - b_n \|^2 - 2\kappa \lambda_n \|c_n - s^* \|^2. \nonumber
\end{equation}
Note that $\frac{\lambda_n}{\lambda_{n+1}} \to 1,$ then $1 - \frac{\mu^2\lambda_n^2}{\lambda_{n+1}^2} \to 1 - \mu^2 >0$ and $\kappa\lambda_n \to \kappa\lambda.$ Also, since $\lambda \geq \min \left\{ \lambda_0, \frac{\mu}{L}, \right\}$ and $\epsilon = \frac{1}{2}\min \left\{ 1-\mu^2, \frac{2\mu\kappa}{L}\right\},$ then for some $n \geq N,$ we get
\begin{eqnarray}
    \|e_n - s^* \|^2 &\leq& \|b_n-s^*\|^2 - \epsilon(\|c_n - b_n\|^2 + \|c_n - s^* \|^2) \nonumber\\
    &\leq& \|b_n - s^* \|^2 - \epsilon\|b_n - s^*\|^2 = (1-\epsilon)\|b_n - s^*\|^2. \label{eq11b}
\end{eqnarray}
Hence
\begin{eqnarray}
    \|s_{n+1}-s_n\|^2 &\leq& (1-\rho)\|a_n - s^*\|^2 + \rho(1-\epsilon)\|b_n - s^*\|^2 - \frac{1-\rho}{\rho}\|s_{n+1}-a_n \|^2 \nonumber\\
    &\leq& \left( (1-\rho)(1+\alpha_n) + \rho(1-\epsilon)(1+\beta_n) \right)\|s_n - s^*\|^2 \nonumber\\
    && - \left( (1-\rho)\alpha_n + \rho(1-\epsilon)\beta_n \right)\|s_{n-1}-s^*\|^2 + ((1-\rho)(1+\alpha_n)\alpha_n +\nonumber\\
    &&\rho(1-\epsilon)(1+\beta_n)\beta_n - (\alpha_n^2 - \alpha_n)\frac{1-\rho}{\rho})\|s_n - s_{n-1}\|^2 \nonumber\\
    && - \frac{1-\rho}{\rho}(1- \alpha_n)\|s_{n+1}-s_n\|^2.\nonumber
\end{eqnarray}
This implies that
\begin{align}
    &\|s_{n+1}-s^*\|^2 + \frac{1-\rho}{\rho}(1-\alpha_n)\|s_{n+1}-s_n\|^2 \nonumber\\
    & \qquad\leq \left( (1-\rho)(1+\alpha_n) + \rho(1-\epsilon)(1+\beta_n) \right)\|s_n - s^*\|^2 \nonumber\\
    & \qquad + ((1-\rho)(1+\alpha_n)\alpha_n + \rho(1-\epsilon)(1+\beta_n)\beta_n - (\alpha_n^2 - \alpha_n)\frac{1-\rho}{\rho})\|s_n - s_{n-1}\|^2. \label{s5}
\end{align}
Let $\xi_n = \left( (1-\rho)(1+\alpha_n) + \rho(1-\epsilon)(1+\beta_n) \right)$ and $\gamma_n =  ((1-\rho)(1+\alpha_n)\alpha_n + \rho(1-\epsilon)(1+\beta_n)\beta_n - (\alpha_n^2 - \alpha_n)\frac{1-\rho}{\rho}).$ Observe that $\xi_n <1$ and $\frac{\gamma_n}{\xi_n} < 1$. Also, since $0 < \alpha_n \leq \frac{1-2\rho}{1-\rho}$, thus $\frac{1-\rho}{\rho}(1-\alpha_n)>1$. 
 Then, from \eqref{s5}, we have
\begin{eqnarray}
    \|s_{n+1}-s^*\|^2 + \|s_{n+1}-s_n\|^2 &\leq& \|s_{n+1} - s^*\|^2 + \frac{1-\rho}{\rho}(1-\alpha_n)\|s_{n+1} - s_n\|^2 \nonumber\\
    &\leq& \gamma_n (\|s_n - s^* \|^2 + \|s_n - s_{n-1}\|^2) \quad \forall n \geq N. \nonumber
\end{eqnarray}
This implies that
\begin{eqnarray}
    \|s_{n+1}-s^*\|^2 &\leq& \|s_{n+1}-s^*\|^2 + \|s_{n+1}-s_n\|^2 \nonumber\\
    &\leq& \gamma_n (\|s_n - s^* \|^2 + \|s_n - s_{n-1}\|^2) \nonumber\\
    &\vdots& \nonumber\\
    &\leq& \gamma_{N+1} (\|s_{N+1}-s^*\|^2 + \|s_{N+1}-s_{N}\|^2) \nonumber\\
    &\leq& \gamma_{N+1}C, \nonumber
\end{eqnarray}
for some $C >0$, which implies that $\{s_n\}$ converges R-linearly to $s^*.$
\end{proof}

\begin{remark}(Inertial versus relaxation) The inequality $0 \leq \beta_n \leq \alpha_n \leq \min \left\{ \frac{1-2\rho}{1-\rho}, \frac{\rho}{1-\rho}, \frac{\rho\epsilon}{1-\rho\epsilon} \right\}$ illustrates the essential balance between the inertial and relaxation parameters. This expression closely resembles the one derived by Attouch and Cabot \cite[Remark 2.13]{Att&Cabot}, with the only distinction being the inclusion of an additional factor that incorporates the parameter $\epsilon = \frac{1}{2}\min \left\{1-\mu^2, \frac{2\mu\kappa}{L} \right\}.$
To clarify this concept, let's temporarily set $\beta_n =\alpha_n = \alpha$. In this case, the upper bound for the inertial parameter becomes $\bar{\alpha}(\rho,\epsilon) = \min \left\{ \frac{1-2\rho}{1-\rho}, \frac{\rho}{1-\rho}, \frac{\rho\epsilon}{1-\rho\epsilon} \right\}.$ The smallest value for the limit of the sequence of inertial parameters is achieved when $\rho = \{0,0.5\}$ and $\alpha_{\min} = 0.$ However, $\bar{\alpha}(\rho,\epsilon) \nearrow \alpha_{\max}\in (0,1)$ as it depends on the specific values of $\rho$ and $\epsilon.$ You can observe the trade-off between the inertial and relaxation parameters for two particular choices of $\epsilon$ in Figure \ref{fig_trade}.

\begin{figure}[h!]
    \centering
    \includegraphics[width = 7.0cm]{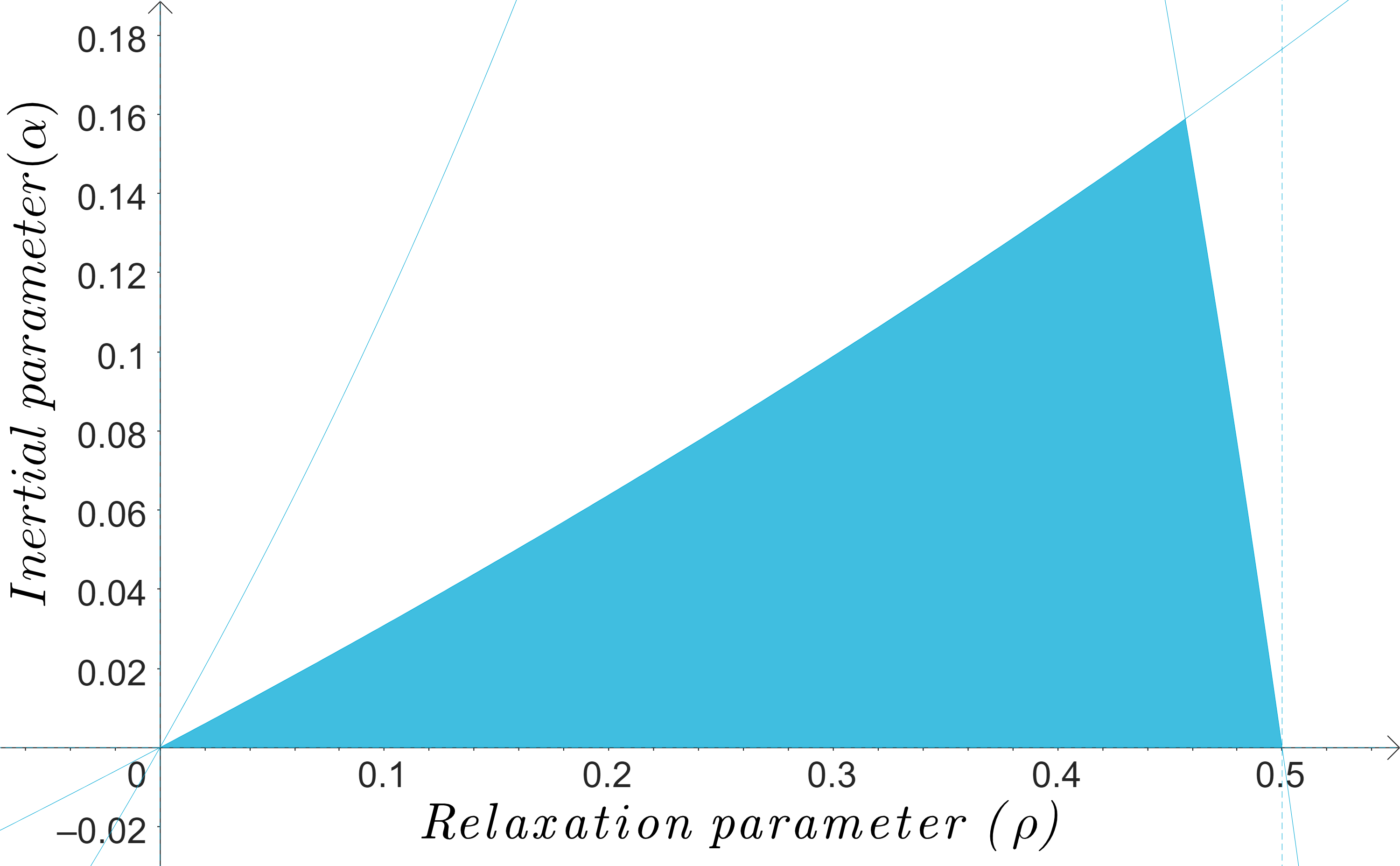}
     \includegraphics[width = 7.0cm]{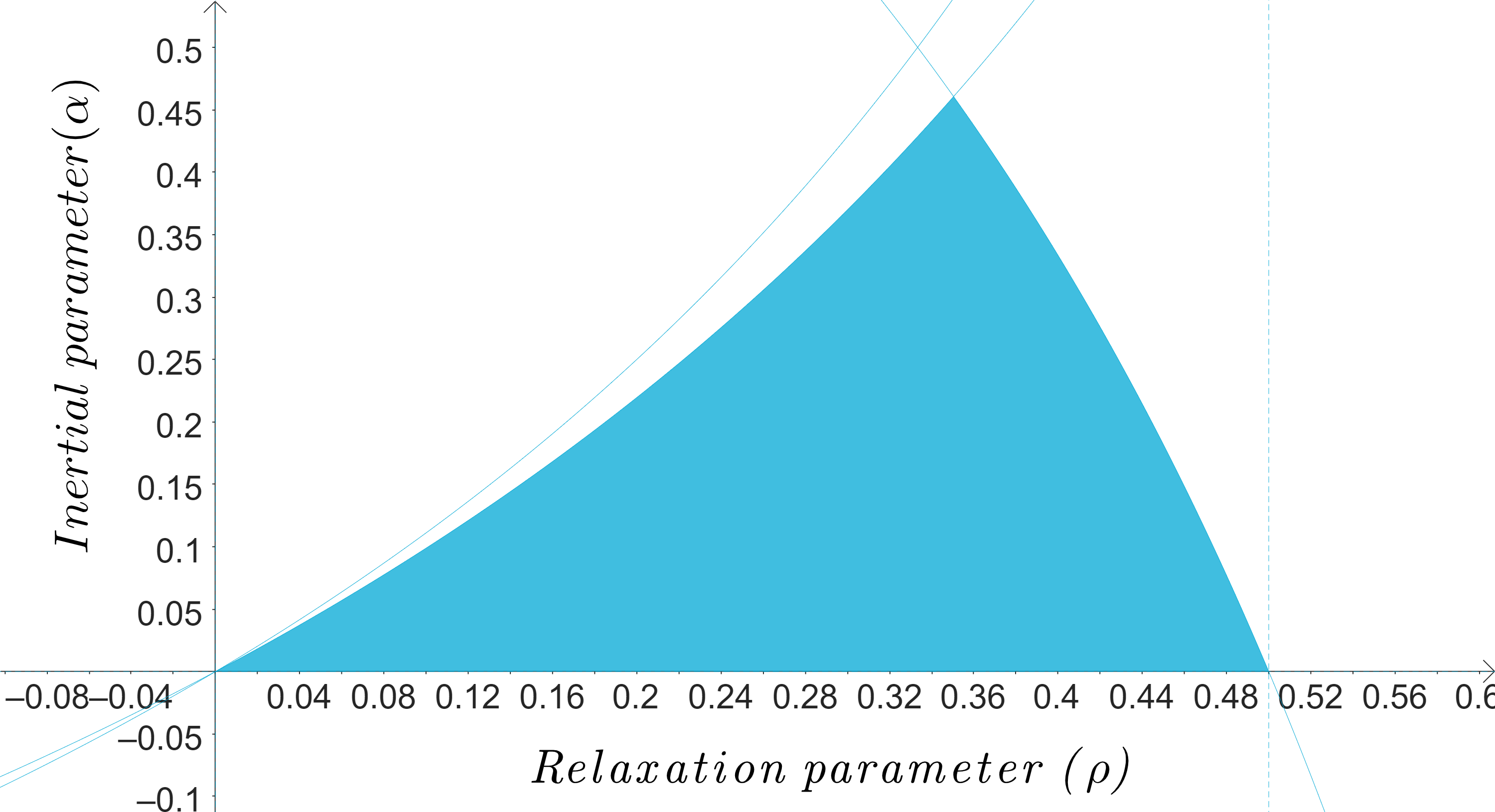}
    \caption{Trade-off between inertial and relaxation parameters for $\epsilon = 0.3$ (Left) and $\epsilon = 0.9$ (Right).}
    \label{fig_trade}
\end{figure}
    
\end{remark}
 
 \section{Numerical experiments}
 In this section, in order to illustrate the computational performance of Algorithm \ref{game}, we divided the numerical experiments into two subsections. The first subsection gives a simulation test for a particular example of a variational inequality problem and the second subsection is devoted to an application of the algorithm to regression problems. The code for the first experiment is written by MATLAB 2022b and the code for the second experiment is implemented using a Jupyter notebook on a PC with Intel Core i7, 3.3 GHz Dual-Core processor, and 16GB RAM. The performance of Algorithm \ref{game} is compared with DIEM, IREM, REM and EM which were derived from Algorithm \ref{game} (see Remark \ref{Rmk1}) corresponding to instances of the proposed method in the literature.

\subsection{Comparison of algorithms}\label{firstExp}
In the experiment, we would like to investigate the effect of the parameters on the performance of the proposed Algorithm \ref{game}. To do this, we consider the following classical academic example of variational inequalities and apply the proposed algorithm.
\begin{experiment}\label{ex1}
    Let $F:\R^N \to \R^N$ be defined by $F(x) = \bar{M}x + \xi,$ with $\bar{M} = RR^\top + S+ D$, where $R$ is a $N \times N$ matrix, $S$ is a $N \times N$ skew-symmetric matrix and $D$ is a $N \times N$ diagonal matrix whose diagonal entries are positive. Consequently, $\bar{M}$  is a positive define matrix, and $\xi$ is any vector in $\R^N$. We define the set $\mathcal{K}$ by $$\mathcal{K}:= \{x \in \R^N : \; Ax \leq b\},$$
    where $A$ is a matrix of size $L \times N$ and $b \in \R_{+}^L$. It is easy to show that $F$ is monotone and Lipschitz continuous with Lipschitz constant $\|\bar{M}\|. $ Also, for $\xi = 0,$ the unique solution of the corresponding variational inequalities is $x^\dagger =0.$ 

    Now, we generate the matrix $N,S, D$ randomly such that $S$ is skew-symmetric and $D$ is a positive definite diagonal matrix. The starting points as well as $A$ and $b$ are also generated randomly. Since our intention is to investigate the effect of the inertial parameters on the algorithm, we consider different values of $\rho, \alpha_n$ and $\beta_n$ corresponding to algorithm GAME, DIEM, IREM, REM and EM discussed in Remark \ref{Rmk1} (see Table \ref{tab_para}). More so, the method is tested using four different values of $N$ and $L$. We choose $\mu = 0.4, \lambda_0 = 0.01$ and $\zeta_n = \frac{1}{10n+9}.$ The algorithm is terminated when $\|b_n - c_n \| < 10^{-6}.$ The computational results are shown in Table \ref{tab_comp} and Figure \ref{fig_comp}.

    \begin{table}[h!]
        \begin{center}
        \caption{Value of parameter $\rho$, $\beta_n$ and $\alpha_n$ used for Example \ref{ex1}.}
        \begin{tabular}{l l l l}
        \toprule
            DIEM & $\rho = 1$ & $\alpha_n = 0.5$ & $\beta_n = 0.2$ \\
            \hline
              IREM & $\rho = 0.6$ & $\alpha_n = 0$ & $\beta_n = 0.2$ \\ 
              \hline
              REM & $\rho = 0.6$ & $\alpha_n = 0$ & $\beta_n = 0$ \\ 
              \hline
              EM & $\rho = 1$ & $\alpha_n = 0$ & $\beta_n = 0$ \\ 
              \hline
              GAME & $\rho = 0.6$ & $\alpha_n = 0.5$ & $\beta_n = 0.2$ \\ 
              \bottomrule
        \end{tabular}
        \end{center}
        \label{tab_para}
    \end{table}

    \begin{figure}[h!]
    \centering
    \includegraphics[width = 7.0cm]{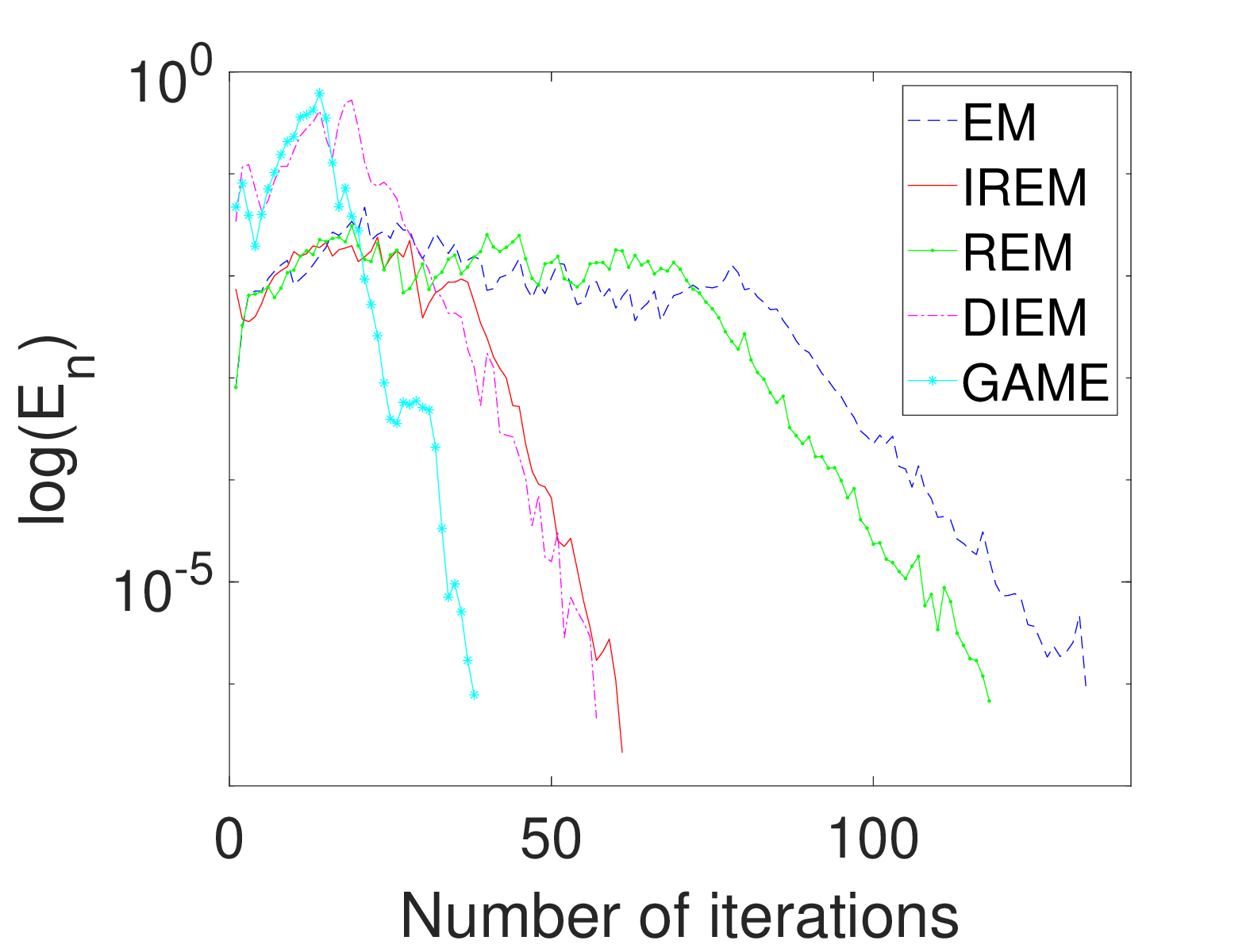}
     \includegraphics[width = 7.0cm]{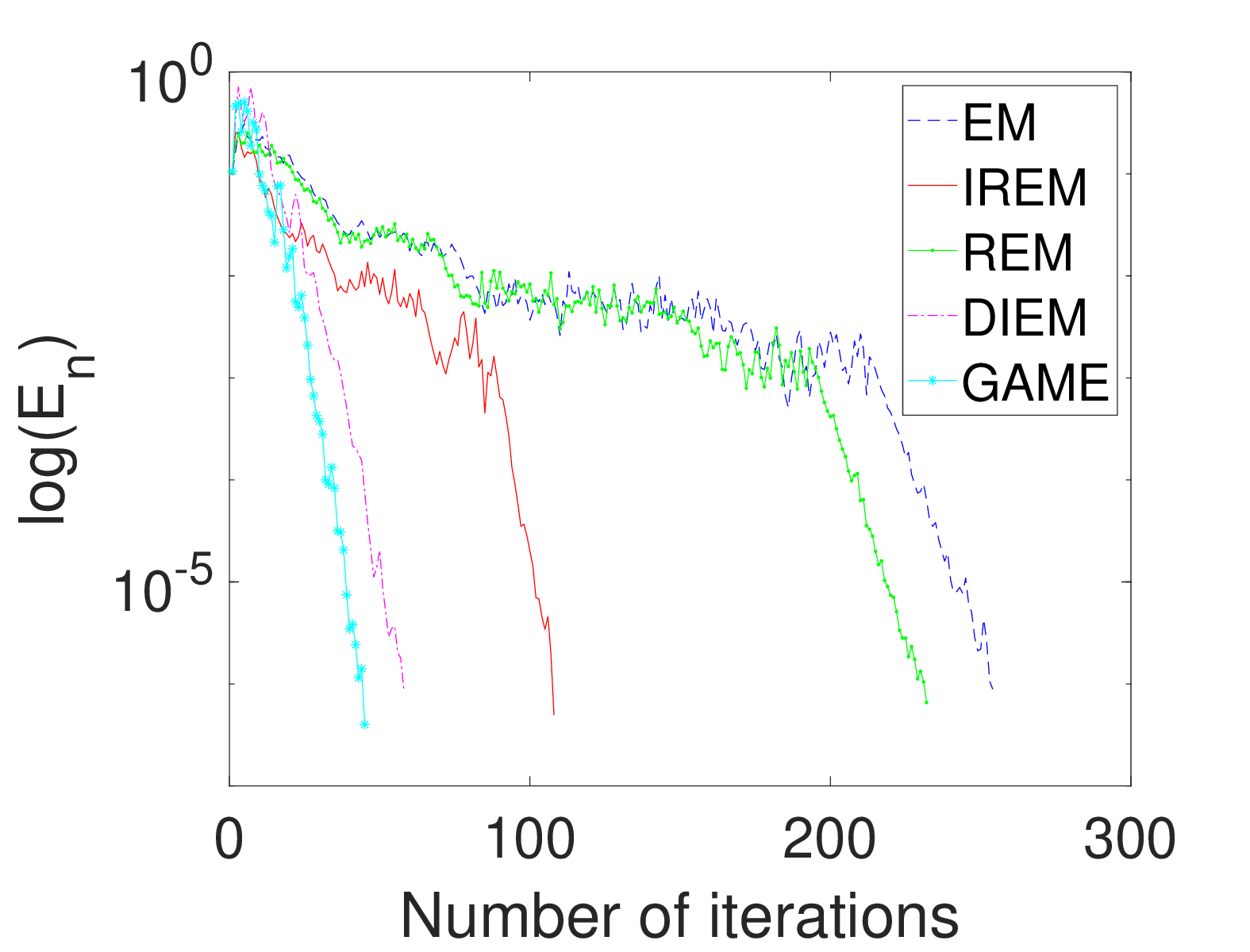}\\
     \includegraphics[width = 7.0cm]{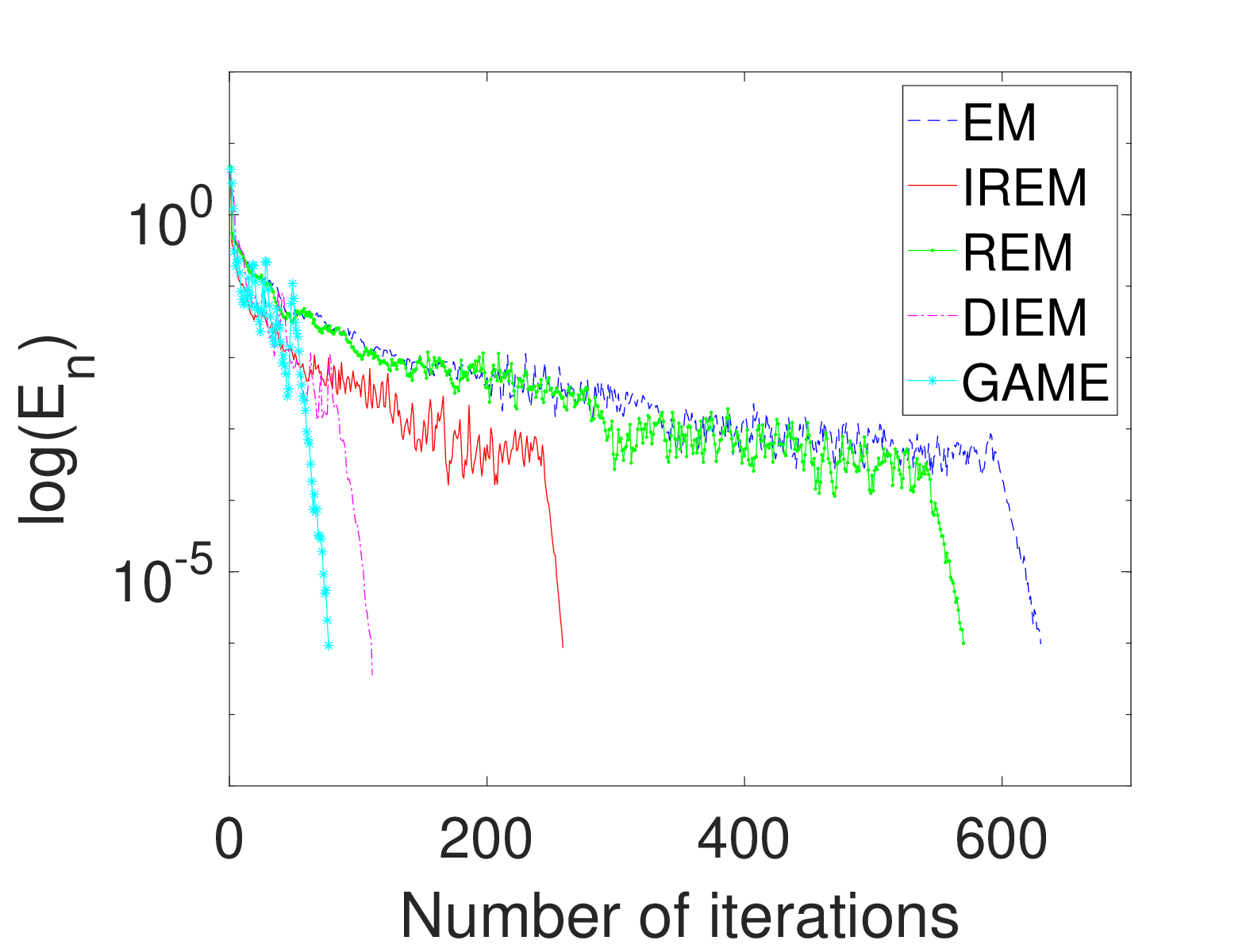}
     \includegraphics[width = 7.0cm]{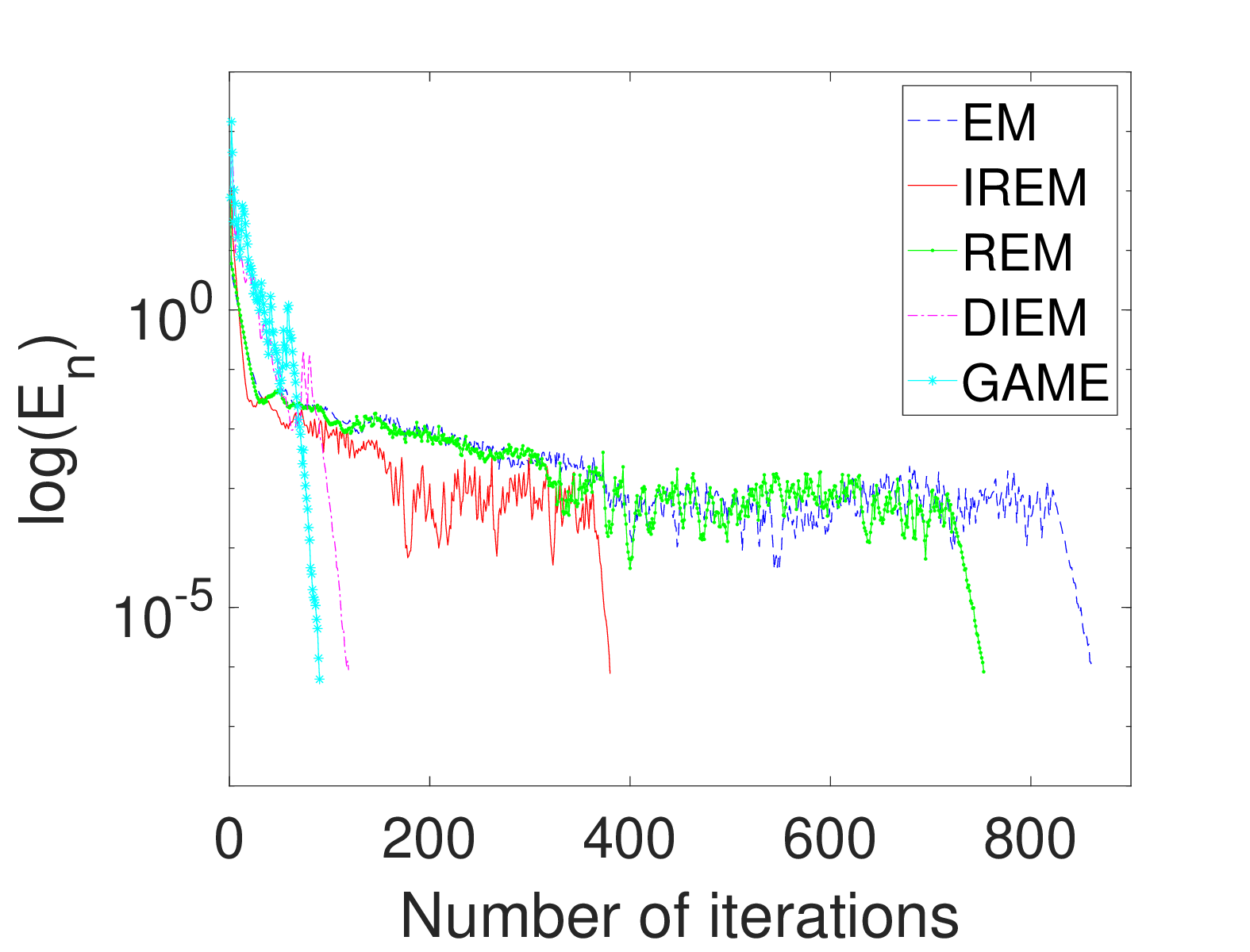}
    \caption{Behaviour of the term $``En = \|c_n - b_n \|"$ for four different values of $N$ and $L$. Top Left: $N = 10, L=5$, Top Right: $N = 20, L=10$, Bottom Left: $N=30, L=15$ and Bottom Right: $N=50, L = 20.$}
    \label{fig_comp}
\end{figure}

\begin{table}[h!]
    \centering
    \begin{tabular}{c| c c c c c c}
    \toprule
       Cases  & \noindent & EM & IREM & REM & DIEM & GAME   \\
       \midrule
       $N = 10, L = 5$ &Iter & 133 & 61 & 118 & 57 & 38 \\
        & Time & 0.0157 & 0.0051 & 0.0057 & 0.0042 & 0.0023 \\
        \hline
       $N = 20, L = 10$ &Iter & 254 & 108 & 232 & 58 & 45 \\
        & Time & 0.0228 & 0.0071 & 0.0114 & 0.0051 & 0.0039 \\
        \hline
       $N = 30, L = 15$ &Iter & 630 & 259 & 570 & 114 & 77 \\
        & Time & 0.0706 & 0.0216 & 0.0415 & 0.0102 & 0.0080 \\
         \hline
       $N = 50, L = 20$ &Iter & 861 & 380 & 753 & 119 & 90 \\
        & Time & 0.1134 & 0.0451 & 0.0823 & 0.0153 & 0.0126 \\
       \bottomrule
    \end{tabular}
    \caption{Computational results of iteration number and CPU time (sec) for Example \ref{ex1}.}
    \label{tab_comp}
\end{table}
    
\end{experiment}

\subsection{Training of ELM based on Algorithm GAME \ref{game}}
ELM exhibits exceptionally rapid training speed and possesses superior pattern reconstruction capabilities compared to backpropagation-based learning. The model's formulation, as expressed in the optimization problem detailed in \eqref{min}, can be effectively tackled through variational inequality techniques. Several approaches are available for solving the problem \eqref{min}, including the utilization of the Moore-Penrose inverse, orthogonal projection, and proper decomposition as mentioned in reference \cite{SunYang}. However, in practical scenarios, the number of independent variables, denoted as $L$, in the linear system $\mathbf{H}\beta = \textbf{X}$ typically exceeds the number of data points, denoted as $N$, which can lead to issues related to overfitting. Conversely, when the number of hidden neurons $L$ is limited, the model's pattern reconstruction capacity tends to be diminished. Additionally, as the size of the training dataset increases, the computational burden associated with computing the inverse matrix $\mathbf{H}^\dagger$ escalates significantly. To address these challenges, various regularization methods have been introduced in the literature. Among these, the two most commonly employed classical techniques are ridge regression, as detailed in reference \cite{Tikho}, and sparse regularization, as outlined in reference \cite{Tibshi}. Sparse regularization plays a pivotal role in feature selection, a technique that has demonstrated remarkable success in the field of machine learning. This study introduces an unsupervised feature selection framework that seamlessly integrates sparse optimization and signal reconstruction into pattern recognition models. This optimization takes the form
\begin{equation*}
    \underset{\beta}{\min}\; \|\mathbf{H}\beta - X\|_2^2 +  \|\beta\|_{1}.
\end{equation*}
From the description given above, we see that the estimation of $\beta$ and the time required are the main challenges for training the ELM. Our aim in this subsection is to implement the proposed algorithm which reduces the computational time and maintains the accuracy of classical algorithms such as FISTA algorithm \cite{BeckFISTA} for training neural networks. The optimization approach \eqref{min} is also interesting because it can be extended to other regularization terms, which control the variable selection by calculating a derivative in the weakest sense through convex envelopes. In addition, the projection operator $P_{\mathcal{K}}$ is computed by the shrinkage operator $shrink(s,\rho)$ \cite{BeckFISTA}, whose $i$-th element is calculated as follows:
\begin{equation*}
    shrink(s,\rho)_{i} =  shrink(s_{i},\rho) = sign(s_{i})\max \{ |s_{i}|-\rho,0\}, 
\end{equation*}
where $s$ is the initial estimation of the weight of the $\beta$-output layer and $\rho >0.$ The pseudo-code for training of the ELM with the proposed Algorithm \ref{game} is written below.

\begin{algorithm} \label{pseudo_game}
  \caption{Training of the ELM with Algorithm \ref{game}}
  \begin{algorithmic}
  \STATE {\bf Require:}  training set $\{s_{i}\}_{i}^{N}$, maximum number of iterations $K$, stopping criterion $\mathcal{P}$, number of neurons $L$, parameters for Algorithm \ref{game}.
  \begin{itemize}
      \item[1.] The weights $\mathbf{W}$ and biases $\mathbf{b}$ in the hidden layer are randomly and independently assigned, with an orthogonal approach.
      \item[2.] The matrix $\mathbf{H}$ is computed by evaluating the function $\mathbf{g}$ using the terms $(s_{i},\mathbf{W},\mathbf{b}).$
      \item[3.] \textbf{while} $\mathcal{P}$ does not occur \textbf{do}
      \item[4.] Determine the optimal solution for \eqref{min} using the variational inequality approach with the utilization of Algorithm \ref{game}.
      \item[5.] \textbf{end while}
  \end{itemize}
  \end{algorithmic}
\end{algorithm}

In each experiment, the training and test data set were standardized to the range [0,1] using the following equation:
In the hidden layer, the weights and biases are random values generated by means of a uniform distribution in the ranges [-1,1] and [0,1], and the sigmoid $\mathbf{g}(x) = \frac{1}{1+e^{-x}}$ was used as the activation function. The algorithm \ref{game} as well as DIEM, IREM, REM, and EM has the following parameters:


\subsubsection{Description of datasets}
To examine the performance of Algorithm \ref{game}, we used the Boston Housing, Autompg, Bodyfat, Bike sharing and Diabetes datasets which were obtained from different free online repositories. These datasets are briefly described below.
\begin{enumerate}
    \item Boston Housing: This dataset contains 506 rows and 14 columns. The dataset was obtained from Kaggle website and is derived from information collected by the U.S. Census Service concerning housing in the area of Boston, MA.
    \item Autompg: This dataset contains 398 rows and 9 columns. The dataset was obtained from Kaggle website and concerns city-cycle fuel consumption in miles per gallon.
    \item Bodyfat: This dataset contains 252 rows and 15 columns. The data was obtained from kaggle, and it estimates the percentage of body fat determined by underwater weighing for 252 men.
    \item Bike sharing: This dataset contains information about the daily count of rental bikes between 2011 and 2012. The dataset contains 731 rows and 16 columns, and it was obtained from UCI machine learning repository.
    \item Diabetes: This dataset contains 768 rows and 9 columns. The dataset was obtained from Kaggle, and it's objective is to predict on diagnostic measurements whether a patient has diabetes.

\end{enumerate}
Table \ref{dataset} shows the number of training and test data, the number of categories and the size of the input vectors for each of the datasets.
\begin{table}[h!]
    \centering
    \begin{tabular}{ccccc}
         \toprule
         Dataset & \# Training & \# Testing & \# Features \\
         \midrule
         Boston Housing &404&102 & 7 \\
         Autompg &318 &80 & 7 \\
         Bodyfat & 201&51 & 14 \\
         Bike sharing &584 & 147& 15 \\
         Diabetes &614 &154 &8  \\
         \bottomrule
    \end{tabular}
    \caption{Information regarding the datasets}
    \label{dataset}
\end{table}

\subsubsection{Description of evaluation metrics}
To analyze the performance of the proposed algorithm and {other comparison algorithms, we used standard evaluation metrics commonly used to evaluate the performance of machine learning methods under regression tasks. The metrics include} the root mean square $(RMSE)$,  mean absolute error (MAE), R$^2$ score, the ratio of the sum squared error (SSE) to the sum squared deviation of the sample SST $(SSE/SST)$ and the ratio between the interpretable sum deviation SSR and SST $(SSR/SST)$. Their mathematical representation is given by  Table \ref{table_metrics}.
The RMSE and MAE metrics represent the average error between the estimated and actual values.  The lower the values of these measures, the higher the method's efficiency. 
The proportion of the overall variation that the simple linear regression model cannot account for is represented by the SSE/SST ratio. At the same time, SSR/SST is the proportion of explained to total variation.


\begin{table}[tpbh]
    \centering
    \caption{\large{Definition of evaluation metrics for the experiments}}\label{table_metrics}
    \begin{tabular}{cc} 
        \toprule
        Metric & Mathematical Expression \\
        \bottomrule
        RMSE & $\sqrt{ \frac{\sum_{i=1}^{n}(y_{i} - \bar{y}_{i})^2}{n}}$ \\
        MAE & $ \sum_{i=1}^{n} \frac{|y_{i} - \bar{y}_{i}|}{n}$ \\
        SSE/SST & $\dfrac{\sum_{i=1}^{n} (\hat{y}_i - y_i)^2}{\sum_{i=1}^{n} (y_i - \bar{y})^2}$\\
        SSR/SST & $\dfrac{\sum_{i=1}^{n} (\hat{y}_i - \bar{y})^2}{\sum_{i=1}^{n} (y_i - \bar{y})^2}$\\
        \bottomrule
    \end{tabular}
\end{table}
\noindent {Note that in Table \ref{table_metrics}, $y$ and $\bar{y}$ denote the vectors of actual and estimated values, respectively, $y^*$ is the average of the actual values and $n$ is the number of pairs of $y_i$ and $\bar{y}_{i}$.}

\subsubsection{Experiment results}
In addition to the numerical experiment presented in subsection \ref{firstExp}, we conduct several other experiments using real-life datasets from the UCI repository.  These experiments are carried out using a 5-fold cross-validation method, and the average results for each metric are presented. Additionally, the dataset used for these experiments is summarized in Table \ref{dataset}.
From the experiment, Table \ref{tab:RMSE}, \ref{tab:MAE}, \ref{tabSSE} and \ref{tabSSR}  shows the results obtained by evaluating the RMSE, MAE, SSE/SST, and SSR/SST respectively by the methods.

\begin{table}[]
    \caption{\large{Results of RMSE obtained by the Algorithms for five datasets.}}
    \resizebox{\columnwidth}{2.6cm}{
    \begin{tabular}{ccccc}
    \toprule 
    \multicolumn{5}{c}{RMSE} \\
    \midrule 
    Dataset & GAME & FISTA & IREM & REM \\
    \midrule
    Boston Housing & 0.5538 & 0.3885 & 0.5220 & 0.5220 \\
    Autompg & 0.1116 & 0.0920 & 0.1042 & 0.1042 \\
    Bodyfat & 0.1039  &  0.0434& 0.1077 & 0.1077 \\
    Bike sharing & 0.0857 & 0.0105& 0.0533 &0.0533 \\
    Diabetes &0.4182 & 0.4168 & 0.4081 & 0.4081 \\
    \bottomrule
    \end{tabular}}
    \label{tab:RMSE}
\end{table}
After examining the results, it is evident that the proposed GAME demonstrates performance on par with other state-of-the-art algorithms in terms of RMSE. However, when we consider the time required for each algorithm to achieve its RMSE value, it becomes apparent that the proposed GAME outperforms the comparison algorithms significantly. These findings are presented in Table \ref{tabNum}. For example, with the Diabetes dataset as an illustration. The proposed GAME achieved an RMSE of 0.3654 within 113 iterations in just 0.07681 seconds. In contrast, FISTA took over 10,000 iterations and 9.3160 seconds to attain an RMSE score of 0.2716. Additionally, IREM yielded a 0.3222 RMSE score within 1599 iterations and 2.0488 seconds, while REM also recorded a 0.3222 RMSE with 1599 iterations, completed in 2.12139 seconds. All the iterations and times taken by other comparison algorithms are much higher than that of the proposed GAME algorithms. Similar patterns can be observed in other tables. These results highlight that the proposed GAME algorithms demonstrate comparable performance compared with other state-of-the-art algorithms, but the proposed GAME algorithm outperforms the other comparison algorithms when considering the number of iterations and the time required to achieve the desired scores of the evaluation metrics. Detailed information is provided in Table \ref{tabNum}.

\begin{table}[ht]
    \centering
    \caption{\large{Results of MAE obtained by the Algorithms for five datasets.}}
    \begin{tabular}{cccccc}
    \toprule 
    \multicolumn{5}{c}{MAE} \\
    \midrule 
    Dataset & GAME & FISTA & IREM & REM \\
    \midrule
    Boston Housing & 0.2778 & 0.1791 & 0.2641 & 0.2641 \\
    Autompg & 0.0671 & 0.0447 & 0.0615 & 0.0616 \\
    Bodyfat & 0.1088 &  0.0610& 0.0781&  0.0785 \\
    Bike sharing &0.0643 &0.0246 &0.0300 &0.0300 \\
    Diabetes &0.3454 & 0.2716 & 0.3222 &0.3222 \\
    \bottomrule
    \end{tabular}
    \label{tab:MAE}
\end{table}


\begin{table}[]
    \centering
    \caption{\large{Results of SSE/SST obtained by the Algorithms for five datasets.}}
    \begin{tabular}{ccccc}
    \toprule 
    \multicolumn{5}{c}{SSE/SST} \\
    \midrule 
    Dataset & GAME & FISTA & IREM & REM \\ 
    \midrule
    Boston Housing & 0.3125 & 0.1612 & 0.2779 & 0.2779 \\  
    Autompg &  0.2897 & 0.1853 & 0.2744 & 0.2743  \\  
    Bodyfat & 1.5525 &0.0918 &0.3461 & 0.3460 \\  
    Bike sharing & 0.3640 & 0.0323 & 0.0571 & 0.0571\\  
    Diabetes &0.9250 &0.7666 &0.7342 &0.7342 \\  
    \bottomrule
    \end{tabular}
    \label{tabSSE}
\end{table}
         
\begin{table}[]
    \centering
    \caption{\large{Results of SSR/SST obtained by the Algorithms for five datasets.}}
    \begin{tabular}{ccccc}
    \toprule 
    \multicolumn{5}{c}{SSR/SST} \\
    \midrule 
    Dataset & GAME & FISTA & IREM & REM \\ 
    \midrule
    Boston Housing & 0.3125 & 0.1612 & 0.2779 & 0.2779 \\  
    Autompg & 0.2897 & 0.1353 & 0.2744 & 0.2743\\  
    Bodyfat & 1.5525 & 0.0718 & 0.3461 & 0.3460 \\  
    Bike sharing & 0.3640 & 0.0623 &  0.0871 & 0.0871 \\  
    Diabetes & 0.9250 & 0.7666 & 0.7342 & 0.7342  \\ 
    \bottomrule
    \end{tabular}
    \label{tabSSR}
\end{table}

\begin{table}[]
    \caption{\large{Comparison of Number of iterations and time of execution by each algorithm.}}
   \resizebox{\columnwidth}{!}{
    \begin{tabular}{c|cc|cc|cc|cc} 
    \toprule 
    Dataset & \multicolumn{2}{c}{GAME} & \multicolumn{2}{c}{FISTA} & \multicolumn{2}{c}{IREM } & \multicolumn{2}{c}{REM}\\ 
        & Iter & Time & Iter & Time & Iter & Time  & Iter & Time   \\ 
    \midrule
     Boston Housing & 166 & 0.12957 & $>$10000 & 5.69053 & 5826 & 4.67538 & 5830 & 4.80806 \\ 
    Autompg & 148 & 0.08541 & $>$10000 & 4.89166 & 150 & 0.10100 & 150 & 0.10033 \\  
    Bodyfat & 141 &0.05733 & $>$10000  &4.21068 & 406 & 0.25426 & 406 & 0.21856\\  
    Bike sharing & 186 & 0.14609 & $>$10000& 8.54421 &565 & 0.65873 & 565& 0.63204  \\  
    Diabetes &113 & 0.07681 &$>$10000 & 9.31650& 1599 &2.04488 & 1599& 2.12139\\  
    \bottomrule
    \end{tabular}}
    \label{tabNum}
\end{table}

\section{Conclusion}
This paper proposed a general adaptive accelerated method (GAME) for solving the variational inequalities problems which is reformulated from a structured optimization problem obtained in the modeling of extreme learning machine network problems. The algorithm consists of an inertial technique and a self-adaptive stepsize which is updated at every iteration in order to avoid a prior guess of the Lipschitz constant of the smooth function of the model. Typically, such a prior guess is too small or too large which affects the rate of convergence in classical algorithms such as the popular fast iterative shrinkage thresholding algorithm (FISTA). The performance of the proposed method is tested with the FISTA and other state-of-art algorithms which are obtained from the GAME algorithm. The numerical results show that the proposed algorithm has significant advantages over the existing methods in terms of the number of iterations and CPU time while training the algorithms on five datasets.















\noindent {\bf {Competing interest:}} The authors declare that there is not competing interest on the paper.

\noindent {\bf {Authors' contributions:}} All authors worked equally on the results and approved the final manuscript.

 \

\end{document}